\documentclass[12pt,reqno]{amsart}
\usepackage[utf8]{inputenc}
\usepackage[T1]{fontenc}
\usepackage[greek,english]{babel}
\usepackage{amsfonts,amsmath,amssymb,amsthm,hyperref,setspace,tikz,verbatim, charter}
\usepackage{array}
\usepackage[a4paper,top=2cm,bottom=2cm,left=2.2cm,right=2.2cm]{geometry}
\usepackage{longtable}
\numberwithin{equation}{section}

\definecolor{ao(english)}{rgb}{0.0, 0.0, 1.0}

\hypersetup{colorlinks=true, linkcolor=ao(english),citecolor=ao(english)}

\usepackage[normalem]{ulem} 

    \newcommand{\manjil}[1]{\textcolor{blue}{#1}}

\newcommand\mycom[2]{\genfrac{}{}{0pt}{}{#1}{#2}}

\newcommand{\op}{\overline{p}}
\newcommand{\opt}{\overline{OPT}}
\usepackage{color, xcolor}

\newtheorem{theorem}{Theorem}[section]
\newtheorem{conjecture}{Conjecture}[section]
\newtheorem{definition}{Definition}[section]

\newtheorem{lemma}{Lemma}[section]
\newtheorem{remark}{Remark}[section]

\title[Arithmetic Properties for Overpartition $k$-tuples with Odd Parts]{Arithmetic Properties modulo powers of 2 and 3 for Overpartition $k$-Tuples with Odd Parts}

\author[H. Das]{Hirakjyoti Das}
\address[H Das]{Department of Mathematics, B. Borooah College, Guwahati 781007, Assam, India}
\email{hirak@bborooahcollege.ac.in}

\author[M. P. Saikia]{Manjil P. Saikia}
\address[M. P. Saikia]{Mathematical and Physical Sciences division, School of Arts and Sciences, Ahmedabad University, Ahmedabad 380009, Gujarat, India}
\email{manjil@saikia.in}

\author[A. Sarma]{Abhishek Sarma}
\address[A. Sarma]{Department of Basic Sciences and Humanities, Assam Skill University,  Mangaldai 784125, Assam, India}
\email{abhitezu002@gmail.com}

\keywords{integer partitions, overpartitions, Ramanujan-type congruences, modular forms, arithmetic density}

\subjclass[2020]{11P81, 11P83}

\allowdisplaybreaks
\begin{document}

\begin{abstract}
    Recently, Drema and N. Saikia (2023) and M. P. Saikia, Sarma, and Sellers (2025) proved several congruences modulo powers of $2$ for overpartition triples with odd parts. In this paper we study further divisibility properties of overpartition $k$-tuples with odd parts using elementary means as well as properties of modular forms. In particular, we prove several congruences modulo multiples of $3$, and an infinite family of congruences modulo powers of $3$; we also prove some cases of a conjecture of Saikia, Sarma, and Sellers.
\end{abstract}

\maketitle

\section{Introduction}

A partition of a positive integer $n$ is a finite non-increasing sequence of positive integers $\lambda=(\lambda_1, \lambda_2, \ldots, \lambda_k)$ such that the parts $\lambda_i$'s sum up to $n$. For instance, $4,3+1,2+2,2+1+1$ and $1+1+1+1$ are the five partitions of $4$. The number of partitions of $n$ is denoted by $p(n)$, with the convention that $p(0)=1$, and its generating function found by Euler is given by
\[
\sum_{n\geq 0}p(n)q^n=\frac{1}{(q;q)_\infty},
\]
where
\[
(a;q)_\infty:=\prod_{i\geq 0}(1-aq^i), \quad |q|<1.
\] Throughout the paper, we will use the notation $f_{k} := (q^k;q^k)_{\infty}$.

The arithmetic properties of partitions have been studied for a long time and several beautiful congruences satisfied by the partition function have been found. This avenue of study has also trickled down to other classes of partitions. For instance, a widely studied class of partitions are the overpartitions, first introduced by Corteel and Lovejoy \cite{CorteelLovejoy}. An overpartition of a nonnegative integer $n$ is a non-increasing sequence of natural numbers whose sum is $n$, and where the first occurrence (or equivalently, the last occurrence) of a number may be overlined. The eight overpartitions of 3 are
\[3,\bar{3},2+1,\bar{2}+1,2+\bar{1},\bar{2}+\bar{1},1+1+1,~\text{and}~\bar{1}+1+1.\]The number of overpartitions of $n$ is denoted by $\op(n)$ and its generating function is given by
\[
\sum_{n\geq 0}\op(n)q^n=\frac{f_{2}}{f_{1}^2}.
\]

Generalizing the idea of overpartitions, we can define an overpartition $k$-tuple, as was done by Keister, Sellers and Vary \cite{KeisterSellersVary}. An overpartition $k$-tuple of $n$ is a $k$-tuple of overpartitions $(\pi_1, \pi_2, \ldots, \pi_k)$ such that the sum of the parts of the partitions $\pi_i$'s equal $n$. The generating function for the number of overpartition $k$-tuples of $n$, denoted by $\op_k(n)$ is given by
\[
\sum_{n\geq 0}\op_k(n)q^n=\frac{f_{2}^k}{f_{1}^{2k}}.
\]
If we restrict all our parts in such an overpartition $k$-tuple to be odd, then we have an overpartition $k$-tuple $(\xi_1, \xi_2, \ldots, \xi_k)$ of $n$ where all parts of the partitions $\xi_i$'s are odd. The generating function for the number of overpartition $k$-tuples of $n$ with odd parts, denoted by $\opt_k(n)$ is given by
\begin{equation}\label{opt}
    \sum_{n\geq 0}\opt_k(n)q^n=\frac{f_{2}^{3k}}{f_{1}^{2k}f_{4}^{k}}.
\end{equation}
The study of the arithmetic properties of the case $k=1$ was initiated by Hirschhorn and Sellers \cite{HSellers}, which led to a lot of follow-up work by other mathematicians (See, for example, \cite{zchen}, \cite{zMerca}). The case $k=2$ has also been studied, the interested reader can look at the work of Lin \cite{Lin} for some representative work. Kim \cite{zBkim} further studied the case using the theory of modular forms. The case $k=3$ here corresponds to overpartition tuples with odd parts, which were very recently studied by Drema and N. Saikia \cite{DremaSaikia} and M. P. Saikia, Sarma and Sellers \cite{SSS}.

The work by Drema and Saikia \cite{DremaSaikia} focused mostly on finding congruences modulo small powers of $2$ for $\opt_3(n)$. The work of Saikia, Sarma and Sellers \cite{SSS} focused mostly on finding arithmetic properties of $\opt_k(n)$ for $k=3, 4$ and odd values modulo powers of $2$. In the present paper, we extend the study of the arithmetic properties of these functions. In particular, we prove some congruences modulo multiples of $3$ which have not appeared earlier. We also prove an infinite family of congruences modulo powers of $3$. However first, we state an unexpected equality below.
\begin{theorem}\label{thm:equality}
     For all $n\geq 1$, we have
     \begin{align*}
         \opt_4(2n)&=2\cdot \opt_8(n).
     \end{align*}
\end{theorem}
\noindent We give a simple proof of the equality in Section \ref{Proof-thm:equality}. Now, we move ahead to some congruences satisfied by $\opt_3(n)$.

\begin{theorem}\label{thm:3}
    For all $n\geq 0$, we have
    \begin{align}
    \opt_3(3n+1)&\equiv 0 \pmod{6},\label{eq9}\\
    \opt_3(12n+7)&\equiv 0 \pmod{12},\label{eq10}\\
 \opt_3(12n+10)&\equiv 0 \pmod{12},\label{eq11}\\
    \opt_3(3n+2)&\equiv 0 \pmod{18},\label{eq12}\\
        \opt_3(6n+5)&\equiv 0\pmod{36} \label{eq15},\\
    \opt_3(24n+23)&\equiv 0\pmod{144}\label{24n+23}.
    \end{align}
\end{theorem}
\noindent The proof of Theorem \ref{thm:3} is via elementary techniques and is given in Section \ref{sec:el2}. The above list is far from exhaustive. In the next theorem, we state a family of congruences for $\opt_3(n)$ modulo $4$.
\begin{theorem}\label{singlemodcong}
For all $n\geq0$, primes $p\ge 5$, quadratic non-residues $r$ modulo $p$, and $A\in\{0,1,2\}$ such that $Ap\equiv 2r+1\pmod{3}$, we have
    \begin{align}
    \opt_3(3pn+R)&\equiv 0\pmod{4},\label{Family 1 mod 4}
\end{align}
where
\[R=\begin{cases}
    2(Ap+r) & \quad \text{if}~~ 2(Ap+r)<3p,\\
     2(Ap+r)-3p & \quad \text{if}~~ 2(Ap+r)>3p.
\end{cases}\]
\end{theorem}
\noindent We give an elementary proof of this result in Section \ref{sec:radu1}.

Saikia, Sarma and Sellers \cite{SSS} prove several results for $\opt_k(n)$, depending on whether $k$ is even or odd. For instance, one of their result is the following theorem.
\begin{theorem}\cite[Theorem 6]{SSS}\label{thm:k2}
    Let $k=(2^m)r$ with $m>0$ and $r$ odd. Then for all $n\geq 0$ we have
    \[
    \opt_k(n)\equiv 0 \pmod{2^{m+1}}.
    \]
\end{theorem}
\noindent They \cite{SSS} conjecture at the end of their paper a more general result for even values of $k$, which is given below.
\begin{conjecture}\cite[Conjecture 1]{SSS}\label{conj2i}
    For all $i\geq1$, $n\geq 0$ and odd $r$, we have
    \begin{align*}
        \opt_{2^ir}(8n+1)&\equiv0\pmod{2^{i+1}},\\
        \opt_{2^ir}(8n+2)&\equiv0\pmod{2^{2i+1}},\\
        \opt_{2^ir}(8n+3)&\equiv0\pmod{2^{i+3}},\\
        \opt_{2^ir}(8n+4)&\equiv0\pmod{2^{2i+4}},\\
        \opt_{2^ir}(8n+5)&\equiv0\pmod{2^{i+2}},\\
        \opt_{2^ir}(8n+6)&\equiv0\pmod{2^{2i+3}},\\
        \opt_{2^ir}(8n+7)&\equiv0\pmod{2^{i+4}}.
    \end{align*}
\end{conjecture}
\noindent We prove three cases of this conjecture here. In fact, we are able to give a better result for one case.
\begin{theorem}\label{thm:conj}
     For all $i\geq1$, $n\geq 0$ and odd $r$, we have
    \begin{align}
    \opt_{2^ir}(8n+1)&\equiv0\pmod{2^{i+1}},\label{conjp-1}\\
    \opt_{2^ir}(4n+3)&\equiv0\pmod{2^{i+3}},\label{conjp-2}\\
    \opt_{2^ir}(8n+5)&\equiv0\pmod{2^{i+2}}.\label{conjp-3}
    \end{align}
\end{theorem}
\noindent The proof of Theorem \ref{thm:conj} is via elementary means and is given in Section \ref{sec:conj}.

\begin{remark}
    It was communicated to us that the last congruence in Conjecture \ref{conj2i} has been recently proven by Qi, Sang, Yao, and Zhou \cite{QiSangYaoZhou}.
\end{remark}

We now move towards proving arithmetic properties of $\opt_k(n)$ modulo powers of $3$. So far, no congruence for modulo arbitrary powers of $3$ are known for general values of $k$. We fill in this gap via the next theorem.
\begin{theorem}\label{thm:KSV:mod3}
For all $i\geq1$ and $n\geq 0$, we have
    \begin{align}
    \opt_{3^i}(3n+2)&\equiv0\pmod{3^{i+1}}.\label{KSV:mod3:Cong1}
    \end{align}    
\end{theorem}
\noindent The proof of Theorem \ref{thm:KSV:mod3} is via elementary means and is given in Section \ref{sec:KSV:mod3}.

Several authors (see \cite{GoswamiJha} and \cite{Singh} for two recent examples which have motivated our results below) have over the years also studied divisibility and density properties of various types of partition functions using the theory of modular forms. We close our results by giving some representative examples of such results in Theorems \ref{thm:mf-1} and \ref{thm:m2} below.
\begin{theorem}\label{thm:mf-1}
    Let $k,n\geq 0$ be integers, for each $i$ with $1\leq i\leq k+1$, if $p_i\geq 3$ is a prime such that $p_i\not\equiv 1\pmod 8$, then for any integer $j\not\equiv 0\pmod{p_{k+1}}$ we have
    \[
    \opt_3(8p_1^2p_2^2\ldots p_{k+1}^2n+p_1^2p_2^2\ldots p_k^2p_{k+1}(8j+p_{k+1})+1)\equiv 0\pmod{8}.
    \]
\end{theorem}
\noindent There are other results of the same type as Theorem \ref{thm:mf-1}. Since the proof techniques are similar, we do not state them here. The proof of Theorem \ref{thm:mf-1} is given in Sections \ref{sec:mf}.

For a fixed positive integer $k$, Gordon and Ono \cite{GordonOno} proved that $p(n)$ is divisible by $2^k$ for almost all $n$. Several mathematicians have done similar work related to other classes of partitions, for instance by Ray and Barman \cite{RayBarman}, Veena and Fathima \cite{VeenaFathima}, etc. It can be shown easily \cite[Eq. (43)]{DremaSaikia} that
\[
\opt_3(2n+1)\equiv 0 \pmod 2,
\]
for all $n\geq 0$. The authors in \cite{SSS} have found several new congruences satisfied by the $\opt_{2k+1}(n)$ function modulo powers of $2$. Motivated by this, we look at the case for $\opt_3(n)$ and study its properties.


\begin{theorem}\label{thm:m2}
    Let $k$ be a positive integer and $p$ be a prime with $p\neq 3$, then $\opt_3(n)$ is almost always divisible by $p^{k}$, that is
    \[
    \lim\limits_{X\rightarrow \infty}\frac{|\{n\leq X~:~\opt_3(n)\equiv 0 \pmod{p^{k}}\}|}{X}=1.
    \]
\end{theorem}
\noindent The proof of Theorem \ref{thm:m2} is given in Section \ref{proof:density}.

The paper is structured as follows. In Section \ref{sec:prelim} we state some preliminary results which are using for our proofs, Sections \ref{Proof-thm:equality} -- \ref{proof:density} contains the proofs of our results, we end our paper with some concluding remarks (and conjectures) in Section \ref{sec:conc}.

\section{Preliminaries}\label{sec:prelim}

\subsection{Elementary Results}\label{sec: elementary}

It can be shown easily \cite[Eq. (43)]{DremaSaikia} that
\[
\opt_3(2n+1)\equiv 0 \pmod 2,
\]
for all $n\geq 0$. Drema and Saikia \cite[Eq. (88)]{DremaSaikia} have also shown that
\begin{align}\label{ds-1}
    \opt_3(3n+1) &\equiv \opt_3(3n+2) \equiv 0 \pmod{3}.
\end{align}
Now, working modulo 2, we have
\begin{align}\label{eq:eq}
     \sum_{n\geq 0}\opt_3(n)q^n \equiv1\pmod{2}.
\end{align}
Hence, for $n\geq1$, we have
\begin{align}
    \opt_3(n) \equiv 0 \pmod{2}.\label{n2}
\end{align}
The last congruence is reminiscent of the following congruence for overpartitions
\[
\op(n)\equiv 0 \pmod 2, \quad \text{for all}~n>0.
\]

Some known 2-, 3-dissections (see for example \cite[Lemmas 2, 3  and 4]{matching}) are stated in the following lemma, which will be used subsequently.
\begin{lemma}\label{lem-diss}
We have
\begin{align}
f_{1}^2 &= \frac{f_{2}f_{8}^5}{f_{4}^2f_{16}^2} - 2 q \frac{f_{2} f_{16}^2}{f_{8}},\label{disf1^2}\\
\frac{1}{f_{1}^2} &= \frac{f_{8}^5}{f_{2}^5 f_{16}^2} + 2 q \frac{f_{4}^2 f_{16}^2 }{f_{2}^5 f_{8}},\label{dis1byf1^2}\\
f_{1}^4 &= \frac{f_{4}^{10}}{f_{2}^2f_{8}^4} - 4 q \frac{f_{2}^2 f_{8}^4}{f_{4}^2},\label{disf1^4}\\
\frac{1}{f_{1}^4} &= \frac{f_{4}^{14}}{f_{2}^{14} f_{8}^4} + 4 q \frac{f_{4}^2 f_{8}^4}{f_{2}^{10}},\label{dis1byf1^4}\\
f_{1}f_{2} &= \frac{f_{6}f_{9}^4}{f_{3}f_{18}^2} - qf_{9}f_{18} - 2q^2 \frac{f_{3}f_{18}^4}{f_{6}f_{9}^2},\label{disf1f2}\\
f_{1}^3 &= \frac{f_{6}f_{9}^6}{f_{3}f_{18}^3}-3qf_{9}^3+4q^3 \frac{f_{3}^2f_{18}^6}{f_{6}^2f_{9}^3},\label{disf1^3}\\
&=a_3 f_3-3 qf_9^3,\label{3df1^3}\\
\frac{1}{f_1^3}&=a_3^2\frac{f_9^3}{f_3^{10}}+3 a_3 q \frac{f_9^6}{f_3^{11}}+9 q^2\frac{f_9^9}{f_3^{12}},\label{3d1/f1^3}
\end{align} where $a_n=a\left(q^n\right):=\displaystyle{\sum_{j,k=-\infty}^\infty q^{n\cdot (j^2+jk+k^2)}}$ is one of Borweins' cubic theta functions.
\end{lemma}

We need the following lemma, which is a refined form of a result of Keister, Sellers and Vary \cite[Lemma 7]{KeisterSellersVary}.
\begin{lemma}\label{lem1}
Let $m$ be a non-negative integer, then for all $1\leq n\leq 2^m$, we have
\[
\binom{2^m}{n}2^n\equiv 0 \pmod{2^{m+\left\lfloor\frac{n-1}{2}\right\rfloor+1}}.
\]
\end{lemma}

\noindent In fact, we also need the following lemma, whose proof follows exactly the line of proof as Lemma \ref{lem1}, so we omit the proof here.
\begin{lemma}\label{modified-ksv}
Let $m$ be a nonnegative integer, then for all $1\leq n\leq 3^m$, we have
\[
\binom{3^m}{n}3^n\equiv 0 \pmod{3^{m+\left\lfloor\frac{n-1}{3}\right\rfloor+1}}.
\]
\end{lemma}
\noindent We also need the following results from Hirschhorn and Sellers \cite[Theorems 1.1 and 1.2]{HSellers}.
\begin{lemma}\label{lem2}
    For all $n\geq 1$, we have
    \[
    \opt_1(n)\equiv \begin{cases}
        2\pmod 4 \quad \text{if $n$ is a square or twice a square},\\0\pmod 4\quad \text{otherwise}.
    \end{cases}
    \]
\end{lemma}
\begin{lemma}\label{lem3}
    For all $n\geq 1$, we have
    \[
    \opt_1(n)\equiv \begin{cases}
        2\quad \text{or}\quad 6\pmod 8 \quad \text{if $n$ is a square or twice a square},\\0\quad \text{or}\quad 4\pmod 8\quad \text{otherwise}.
    \end{cases}
    \]
\end{lemma}

\subsection{A primer on modular forms}\label{sec:modularforms}

We recall some aspects of modular forms that will be used in Sections \ref{sec:mf} and \ref{proof:density}. Let $\mathbb{H}$ be the complex upper half-plane. For a positive integer $N$, we define the following matrix groups:
\begin{align*}
\Gamma & :=\left\{\begin{bmatrix}
a  &  b \\
c  &  d      
\end{bmatrix}: a, b, c, d \in \mathbb{Z}, ad-bc=1
\right\},\\
\Gamma_{\infty} & :=\left\{
\begin{bmatrix}
1  &  n \\
0  &  1      
\end{bmatrix} \in \Gamma : n\in \mathbb{Z}  \right\},
\end{align*}
$$\Gamma_{0}(N) :=\left\{
\begin{bmatrix}
a  &  b \\
c  &  d      
\end{bmatrix} \in \Gamma : c\equiv~0\pmod N \right\},$$
$$\Gamma_{1}(N) :=\left\{
\begin{bmatrix}
a  &  b \\
c  &  d      
\end{bmatrix} \in \Gamma_0(N) : a\equiv~d\equiv~1\pmod N \right\}$$
and
$$
\Gamma(N) :=\left\{
	\begin{bmatrix}
		a  &  b \\
		c  &  d      
	\end{bmatrix} \in \textup{SL}_2(\mathbb{Z}) : a\equiv d\equiv 1\pmod N,~ \textup{and}~b\equiv c\equiv 0\pmod N\right\}.
	$$
        A congruence subgroup $\Gamma(N)$ of $\Gamma$ is a subgroup satisfying $\Gamma(N)\subseteq\Gamma$ for some $N$ and the smallest such $N$ s called the level of $\Gamma$. The group
		$$\textup{GL}^{+}_{2}(\mathbb{R}) :=\left\{
	\begin{bmatrix}
		a  &  b \\
		c  &  d      
	\end{bmatrix} : a, b, c, d \in \mathbb{R}~\textup{and}~ad-bc>0\right\}$$
	acts on $\mathbb{H}$ by $\begin{bmatrix}
		a  &  b \\
		c  &  d      
	\end{bmatrix}z = \dfrac{az+b}{cz+d}$. We identify $\infty$ with $\dfrac{1}{0}$. We also define $\begin{bmatrix}
	a  &  b \\
	c  &  d
	\end{bmatrix}\dfrac{r}{s} = \dfrac{ar+bs}{cr+ds}$, where $\dfrac{r}{s}\in\mathbb{Q}\cup\{\infty\}$. This will give an action of $\textup{GL}^{+}_{2}(\mathbb{R})$ on the extended upper half-plane $\mathbb{H^{*}}=\mathbb{H}\cup\mathbb{Q}\cup\{\infty\}$. Suppose that $\Gamma$ is a congruence subgroup of $\textup{SL}_2(\mathbb{Z})$, then a cusp of $\Gamma$ is an equivalence class in $\mathbb{Q}\cup\{\infty\}$.

The group $\textup{GL}_2^{+}(\mathbb{R})$ also acts on the functions $f : \mathbb{H} \to \mathbb{C}$. In particular, suppose that $\gamma = \begin{bmatrix}
a  &  b \\
c  &  d      
\end{bmatrix} \in \textup{GL}_2^{+}(\mathbb{R}) $. If $f(z)$ is a meromorphic function on $\mathbb{H}$ and $\ell$ is an integer, then define the slash operator $\mid_{\ell}$ by 
$(f\mid_{\ell}\gamma )(z) := (det \gamma)^{\ell/2}(cz+d)^{-\ell}f(\gamma z)$.\\

\begin{definition}
Let $\Gamma$ be a congruence subgroup of level N. A holomorphic function $f : \mathbb{H} \to \mathbb{C}$ is called a modular form with integer weight $\ell$ on $\Gamma$ if the following hold:
\begin{enumerate}
    \item We have
    \begin{align*}
        f\left(\frac{az+b}{cz+d}\right)=(cz+d)^{\ell}f(z)
    \end{align*}
    for all $z\in\mathbb{H}$ and all $\begin{bmatrix}
        a&b\\
        c&d
    \end{bmatrix}\in\Gamma$.
    \item If $\gamma\in\textup{SL}_2(\mathbb{Z})$, then $(f\mid_{\ell}\gamma )(z)$ has a Fourier expansion of the form
    \begin{align*}
        (f\mid_{\ell}\gamma )(z)=\sum_{n=0}^{\infty}a_{\gamma}(n)q^n_{N},
    \end{align*}
    where $q^n_{N}=e^{\frac{2\pi iz}{N}}$.
\end{enumerate}
\end{definition}

We denote by $M_\ell(\Gamma_1(N))$ for a positive integer $\ell$, the complex vector space of modular forms of weight $\ell$ with respect to $\Gamma_1(N)$. Also
\begin{align*}
[\Gamma : \Gamma_0(N)]&:=N\prod_{\ell\mid N} \left( 1+\dfrac{1}{\ell}\right),
\end{align*}
where $\ell$ is a prime number.
\begin{definition}\cite[Definition 1.15]{OnoModularity}
    If $\chi$ is a Dirichlet character modulo $N$, then a modular form $f\in M_\ell(\Gamma_1(N))$ has Nebentypus character $\chi$ if 
    $$f\left( \frac{az+b}{cz+d}\right)=\chi(d)(cz+d)^{\ell}f(z)$$ for all $z\in \mathbb{H}$ and all $\begin{bmatrix}
a  &  b \\
c  &  d      
\end{bmatrix}\in \Gamma_0(N)$. The space of such modular forms is denoted by $M_\ell(\Gamma_0(N),\chi)$.
\end{definition}

Recall that the Dedekind's eta-function $\eta(z)$ is defined by
\begin{align*}
	\eta(z):=q^{1/24}(q;q)_{\infty}=q^{1/24}\prod_{n=1}^{\infty}(1-q^n),
\end{align*}
where $q:=e^{2\pi iz}$ and $z\in \mathbb{H}$. A function $f(z)$ is called an eta-quotient if it is of the form
\begin{align*}
f(z)=\prod_{\delta|N}\eta(\delta z)^{r_\delta},
\end{align*}
where $N$ is a positive integer and $r_{\delta}$ is an integer. We now recall two theorems from \cite[p. 18]{OnoModularity} which will be used to prove our result.

\begin{theorem}\cite[Theorem 1.64 and Theorem 1.65]{OnoModularity}\label{thm:ono1} If $f(z)=\prod_{\delta|N}\eta(\delta z)^{r_\delta}$ 
is an eta-quotient with $\ell=\dfrac{1}{2}\sum_{\delta|N}r_{\delta}\in \mathbb{Z}$, 
with
	$$\sum_{\delta|N} \delta r_{\delta}\equiv 0 \pmod{24}$$ and
	$$\sum_{\delta|N} \frac{N}{\delta}r_{\delta}\equiv 0 \pmod{24},$$
	then $f(z)$ satisfies $$f\left( \frac{az+b}{cz+d}\right)=\chi(d)(cz+d)^{\ell}f(z)$$
	for every  $\begin{bmatrix}
		a  &  b \\
		c  &  d      
	\end{bmatrix} \in \Gamma_0(N)$. Here the character $\chi$ is defined by $\chi(d):=\left(\frac{(-1)^{\ell} \prod_{\delta |N}\delta^{r_{\delta}}}{d}\right)$. 
	 In addition, if $c, d,$ and $N$ are positive integers with $d|N$ and $\gcd(c, d)=1$, then the order of vanishing of $f(z)$ at the cusp $\dfrac{c}{d}$ 
	is $\dfrac{N}{24}\sum_{\delta|N}\dfrac{\gcd(d,\delta)^2r_{\delta}}{\gcd(d,\frac{N}{d})d\delta}$.
\end{theorem}
Suppose that $\ell$ is a positive integer and that $f(z)$ is an eta-quotient satisfying the conditions of the above theorem. 
If $f(z)$ is holomorphic at all of the cups of $\Gamma_0(N)$, 
then $f(z)\in M_{\ell}(\Gamma_0(N), \chi)$.

\begin{theorem}\cite[due to Serre, p. 43]{OnoModularity}\label{thm:ono2}
    If $f(z)=\sum_{n=0}^\infty a(n)q^n\in M_\ell(\Gamma_0(N), \chi)$ has Fourier expansion
    \[
    f(z)=\sum_{n=0}^\infty c(n)q^n\in \mathbb{Z}[[q]],
    \]
    then for each positive integer $m$ there exists a constant $\alpha>0$ such that
    \[
    |\{n\leq X~:~c(n)\not\equiv0\pmod m\}|=\mathcal{O}\left(\frac{X}{(\log X)^\alpha}\right).
    \]
\end{theorem}

We recall the definition of Hecke operators.
\begin{definition}
Let $m$ be a positive integer and $f(z) = \sum_{n \geq 0} a(n)q^n \in M_{k}(\Gamma_0(N),\chi)$, where $\chi$ is a Dirichlet character modulo $N$. Then the action of Hecke operator $T_m$ on $f(z)$ is defined by 
\begin{align*}
f(z)|T_m := \sum_{n \geq 0} \left(\sum_{d\mid \gcd(n,m)}\chi(d)d^{k-1}a\left(\frac{nm}{d^2}\right)\right)q^n.
\end{align*}
In particular, if $m=p$ is prime, we have 
\begin{align}\label{hecke}
f(z)|T_p := \sum_{n \geq 0} \left(a(pn)+\chi(p)p^{k-1}a\left(\frac{n}{p}\right)\right)q^n.
\end{align}
We note that $a(n)=0$ unless $n$ is a nonnegative integer.
\end{definition}
\begin{definition}\label{hecke2}
	A modular form $f(z)=\sum_{n \geq 0}a(n)q^n \in M_{k}(\Gamma_0(N),\chi)$ is called a Hecke eigenform if for every $m\geq2$ there exists a complex number $\lambda(m)$ for which 
	\begin{align}\label{hecke3}
	f(z)|T_m = \lambda(m)f(z).
	\end{align}
\end{definition}

In the special case when $\chi$ is the trivial character (i.e., $\chi(n) = 1$ for all $(n,N) = 1$), the Hecke operator simplifies to
\begin{align*}
f(z)|T_m := \sum_{n \geq 0} \left(\sum_{d\mid \gcd(n,m)}d^{k-1}a\left(\frac{nm}{d^2}\right)\right)q^n.
\end{align*}

\section{Proof of Theorem \ref{thm:equality}} \label{Proof-thm:equality}
Using \eqref{dis1byf1^4}, we have
\begin{align*}
    \sum_{n=0}^\infty\opt_4(n)q^n&=\dfrac{f_2^{12}}{f_1^{8}f_4^{4}}=\frac{f_4^{24}}{f_2^{16} f_8^8}+8q\frac{ f_4^{12} }{f_2^{12}}+16q^2\frac{f_8^8}{f_2^8},
\end{align*}
which gives
\begin{align}
    \sum_{n=0}^\infty\opt_4(2n)q^n&=\frac{f_2^{24}}{f_1^{16} f_4^8}+16q\frac{f_4^8}{f_1^8}.\label{For OPT4 and OPT8 1}
\end{align}
Again squaring \eqref{dis1byf1^4} and then replacing $q$ by $-q$, we have
\begin{align*}
    \frac{1}{f_1^8}&=\frac{f_4^{28}}{f_2^{28} f_8^8}+8q\frac{f_4^{16}}{f_2^{24}}+16q^2\frac{f_8^8 f_4^4}{f_2^{20}},\\
    \frac{f_1^8f_4^8}{f_2^{24}}&=\frac{f_4^{28}}{f_2^{28} f_8^8}-8q\frac{f_4^{16}}{f_2^{24}}+16q^2\frac{f_8^8 f_4^4}{f_2^{20}}.
\end{align*}
The two identities above imply
\begin{align*}
   \frac{f_2^{24}}{f_1^{16}f_4^8} &=  1+ 16q\frac{f_4^{8}}{f_1^{8}},
\end{align*}
which together with \eqref{For OPT4 and OPT8 1}, gives
\begin{align*}
    \sum_{n=0}^\infty\opt_4(2n)q^n&=2  \frac{f_2^{24}}{f_1^{16}f_4^8}-1=2\sum_{n=0}^\infty\opt_8(n)q^n-1.
\end{align*}
This equality proves Theorem \ref{thm:equality}. \qed

\section{Elementary proof of Theorem \ref{thm:3}}\label{sec:el2}
From \eqref{opt}, we recall
\begin{align}\label{n}
        \sum_{n\geq 0}\opt_3(n)q^n &=\frac{f_{2}^9}{f_{1}^6f_{4}^3}=  \frac{f_{2}^9}{f_{4}^3} \Bigg(\frac{f_{4}^{14}}{f_{2}^{14} f_{8}^4} + 4 q \frac{f_{4}^2 f_{8}^4}{f_{2}^{10}}\Bigg)\Bigg(\frac{f_{8}^5}{f_{2}^5 f_{16}^2} + 2 q \frac{f_{4}^2 f_{16}^2 }{f_{2}^5 f_{8}}\Bigg),
\end{align}
where we  have employed \eqref{dis1byf1^2} and \eqref{dis1byf1^4}. Now, working modulo $4$, we have,
\begin{align}
   \sum_{n\geq 0}\opt_3(n)q^n & \equiv \frac{f_{2}^7}{f_{4}^3} \Bigg(\frac{f_{8}^5}{f_{2}^5 f_{16}^2} + 2 q \frac{f_{4}^2 f_{16}^2 }{f_{2}^5 f_{8}}\Bigg) \pmod{4}.\label{opt(n)}
\end{align}
Extracting the terms involving $q^{2n+1}$, dividing both sides by $q$ and replacing $q^2$ by $q$, we have
\begin{align}
    \sum_{n\geq 0}\opt_3(2n+1)q^n &\equiv 2\frac{f_{1}^2 f_{8}^2}{f_{2}f_{4}} \equiv 2 f_{2}^6 \pmod{4}\label{opt(2n+14)}.
\end{align}

Now, extracting the terms involving $q^{2n}$ from \eqref{opt(n)} and replacing $q^2$ by $q$, we have 
\begin{align}\label{2n4}
     \sum_{n\geq 0}\opt_3(2n)q^n &\equiv \frac{f_{1}^2f_{4}^5}{f_{2}^3f_{8}^2}\equiv \frac{f_{4}^5}{f_{2}^3f_{8}^2}\Bigg(\frac{f_{2}f_{8}^5}{f_{4}^2f_{16}^2} - 2 q \frac{f_{2} f_{16}^2 }{f_{8}}\Bigg) \pmod{4},
\end{align}
where in the last step, we have employed \eqref{disf1^2}. Equations \eqref{opt(2n+14)} and \eqref{2n4} will be recalled in a while. We now move on to prove the congruences.

\emph{Proof of \eqref{eq9}}. From equations \eqref{ds-1} and \eqref{n2} for $n\geq0$, we have
\begin{align}
    \opt_3(3n+1) \equiv \opt_3(3n+2) \equiv 0 \pmod{6}.
\end{align}
This proves \eqref{eq9}. In fact, $\opt_3(3n+2) \equiv 0 \pmod{18}$. We prove this next. 

\emph{Proof of \eqref{eq12}}. Working modulo $9$, we obtain
\begin{align}
     \sum_{n\geq 0}\opt_3(n)q^n &\equiv \frac{f_{6}^3 }{f_{3}^3 f_{12}^3} f_{1}^3 \Bigg(f_{4}^3\Bigg)^2 \nonumber\\
     &\equiv \frac{f_{6}^3 }{f_{3}^3 f_{12}^3} \Bigg(\frac{f_{6}f_{9}^6}{f_{3}f_{18}^3}-3qf_{9}^3+4q^3 \frac{f_{3}^2f_{18}^6}{f_{6}^2f_{9}^3}\Bigg)\nonumber \\ &\quad \times \Bigg(\frac{f_{24}f_{36}^6}{f_{12}f_{72}^3}-3q^4f_{36}^3+4q^{12} \frac{f_{12}^2f_{72}^6}{f_{24}^2f_{36}^3}\Bigg)^2 \pmod{9},
\end{align}
where in the last step, we have employed \eqref{disf1^3}.

Extracting the terms involving $q^{3n+2}$, dividing both sides by $q^2$ and replacing $q^3$ by $q$, we arrive at
\begin{align}
    \opt_3(3n+2) &\equiv 18\frac{ f_2^3 f_3^3 f_{8} f_{12}^9 }{f_1^3 f_{4}^4 f_{24}^3} q+ 9\frac{ f_2^4 f_3^6 f_{12}^6 }{f_1^4 f_{4}^3 f_{6}^3} q^2+36\frac{ f_2 f_{6}^6 f_{12}^6 }{f_1 f_3^3 f_{4}^3} q^3+ 72\frac{ f_2^3 f_3^3 f_{24}^6 }{f_1^3 f_{4} f_{8}^2}q^{5} \nonumber\\
    &\equiv 0 \pmod{9}.\label{3n+29}
\end{align}
Hence, from \eqref{n2} and \eqref{3n+29}, we obtain 
\begin{align*}
     \opt_3(3n+2) \equiv 0 \pmod{18}.
\end{align*}
This proves \eqref{eq12}.

In the remainder of this section, we will use $\opt_3(3n+1) \equiv 0\pmod{3}$ and \eqref{3n+29} above without commentary.

\emph{Proof of \eqref{eq10}}. Next, we have 
\begin{align}
    \opt_3(12n+7) = \opt_3(4(3n+1)+3) \equiv 0 \pmod{4},\label{12n+74}
\end{align}
using \cite[Theorem 1]{SSS}.
Also, 
\begin{align}
    \opt_3(12n+7) = \opt_3(3(4n+2)+1) \equiv 0 \pmod{3}.\label{12n+73}
\end{align}
From \eqref{12n+74} and \eqref{12n+73}, we conclude
\begin{align*}
     \opt_3(12n+7) \equiv 0 \pmod{12}.
\end{align*}
This proves \eqref{eq10}.

\emph{Proof of \eqref{eq11}}. Again, extracting the terms involving odd powers of $q$ from \eqref{2n4}, we have
\begin{align*}
     \sum_{n\geq 0}\opt_3(4n+2)q^n &\equiv 2\frac{f_{2}^5f_{8}^2}{f_{1}^2f_{4}^3}\equiv 2f_{4}^3 \equiv 2\Bigg(\frac{f_{24}f_{36}^6}{f_{12}f_{72}^3}-3q^4f_{36}^3+4q^{12} \frac{f_{12}^2f_{72}^6}{f_{24}^2f_{36}^3}\Bigg) \pmod{4}.
\end{align*}
Extracting the terms involving $q^{3n+2}$, dividing both sides by $q^2$ and replacing $q^3$ by $q$, we obtain
\begin{align}
    \opt_3(12n+10) \equiv 0 \pmod{4}\label{12n+104}
\end{align}
and 
\begin{align}
    \opt_3(12n+10) = \opt_3(3(4n+3)+1) \equiv 0 \pmod{3}. \label{12n+103}
\end{align}
Combining \eqref{12n+103} and \eqref{12n+104}, we arrive at
\begin{align*}
     \opt_3(12n+10) \equiv 0 \pmod{12}.
\end{align*}
This proves \eqref{eq11}.

\emph{Proof of \eqref{eq15}}. Next, we have
\begin{align}
    \opt_3(6n+5) = \opt_3(3(2n+1)+2) \equiv 0 \pmod{9}.\label{12n+59}
\end{align}
Also, from \eqref{opt(2n+14)}, we recall
\begin{align}
     \sum_{n\geq 0}\opt_3(2n+1)q^n &\equiv 2 f_{4}^3 \pmod{4}.\label{cong6n+51}
\end{align}
Employing  \eqref{disf1^3} in \eqref{cong6n+51} and extracting the odd powered terms of $q$, we obtain
\begin{align*}
         \sum_{n\geq 0}\opt_3(2n+1)q^n &\equiv 2 f_{4}^3 \equiv 2 \Bigg(\frac{f_{24}f_{36}^6}{f_{12}f_{72}^3}-3q^4f_{36}^3+4q^{12} \frac{f_{12}^2f_{72}^6}{f_{24}^2f_{36}^3}\Bigg) \pmod{4}.
\end{align*}
Extracting the terms involving $q^{3n+2}$, we obtain
\begin{align}
    \opt_3(6n+5) \equiv 0 \pmod{4}.\label{12n+54}
\end{align}
Combining $\eqref{12n+54}$ and $\eqref{12n+59}$, we obtain
\begin{align*}
    \opt_3(6n+5) \equiv 0 \pmod{36}.
\end{align*}
This proves \eqref{eq15}.

\emph{Proof of \eqref{24n+23}}. Using \eqref{eq12}, we know that $\opt_3(24n+23)\equiv 0\pmod9$. So, to prove \eqref{24n+23} we only need to check whether the desired congruence is divisible by 16. Using \eqref{n}, we have
\begin{align*}
    \sum_{n=0}^{\infty}\opt_3(n)q^n &= \frac{f_{2}^9}{f_{1}^6f_{4}^3}\equiv \frac{f_{1}^{16}f_{2}}{f_{1}^6f_{4}^3}\equiv \frac{f_{2}}{f_{4}^3} \cdot f_{1}^2 \cdot \left(f_{1}^4\right)^2\\
    &\equiv \frac{f_{2}}{f_{4}^3}  \left(\frac{f_{4}^{10}}{f_{2}^2f_{8}^4} - 4 q \frac{f_{2}^2 f_{8}^4}{f_{4}^2}\right)^2 \left(\frac{f_{2}f_{8}^5}{f_{4}^2f_{16}^2} - 2 q \frac{f_{2} f_{16}^2 }{f_{8}}\right) \pmod{16}.
\end{align*}
Extracting the terms involving $q^{2n+1}$, dividing both sides by $q$ and replacing $q^2$ by $q$ , we have
\begin{align*}
     \sum_{n=0}^{\infty}\opt_3(2n+1)q^n \equiv 8\frac{ f_1^2 f_2^3 f_4^5}{f_{8}^2} - 2 \frac{ f_2^{17} f_{8}^2 }{f_1^2 f_4^9}\pmod{16}.
\end{align*}
Employing \eqref{disf1^2} and \eqref{dis1byf1^2} and extracting the odd powered terms of $q$, we obtain
\begin{align*}
     \sum_{n=0}^{\infty}\opt_3(4n+3)q^n \equiv - 4 \frac{f_{1}^{12} f_4 f_8^2}{f_2^7}\pmod{16}.
\end{align*}
Using \eqref{disf1^4}, we extract the terms involving $q^{2n+1}$ to arrive at
\begin{align*}
      \sum_{n=0}^{\infty}\opt_3(8n+7)q^n &\equiv 48 \frac{f_2^{19}}{f_1^9 f_4^2}+256 q \frac{f_4^{14}}{f_1 f_2^5} \equiv 0\pmod{16}.
\end{align*}
This completes the proof of \eqref{24n+23}. \qed

\section{Proof of Theorem \ref{singlemodcong}}\label{sec:radu1}
Combining the results \cite[Theorem 7]{SSS} and \cite[Theorem 1.1]{HSellers}, we know that \[\opt_3(n)\equiv 0 \pmod{4}\] for all $n\ge 1$ if and only if $n$ is neither a square nor twice a square. So, proving \eqref{Family 1 mod 4} is equivalent to proving that $3pn+R$ is neither a square nor twice a square for all $n\ge0$, primes $p\ge5$, and $R$ defined in Theorem \ref{singlemodcong}. First, note that since $Ap\equiv2r+1\pmod{3}$,
\begin{align*}
    R
& \equiv \begin{cases}
    2(3r+1) \pmod{3}  & \quad \text{if}\ 2(Ap+r)<3p,\\
     2(3r+1)-3p  \pmod{3} & \quad \text{if}\ 2(Ap+r)>3p
\end{cases}\\
& \equiv \begin{cases}
    2 \pmod{3}  & \quad \text{if}\ 2(Ap+r)<3p,\\
     2  \pmod{3} & \quad \text{if}\ 2(Ap+r)>3p.
\end{cases}
\end{align*}
Therefore, $3pn+R\equiv2\pmod{3}$ is never a square for all $n\ge0$, primes $p\ge5$, and $R$.

If $R=2(Ap+r)$, then we have the following two cases depending on the parity of $n$. 
\begin{enumerate}
    \item[1.] When $n=2m$ for some $m$, $3pn+R=2(p(3m+A)+r)$ is not twice a square since $p(3m+A)+r$ can not be a square.
    \item[2.] When $n$ is odd, $3pn+R$ is odd. Therefore, $3pn+R$ is not twice a square.
\end{enumerate}

If $R=2(Ap+r)-3p$, then again we have the following two cases depending on the parity of $n$.
\begin{enumerate}
    \item[1.] When $n$ is even, $3pn+R$ is odd. Therefore, $3pn+R$ is not twice a square.
    \item[2.] When $n=2m+1$ for some $m$, $3pn+R=2(p(3m+A)+r)$ is not twice a square since $p(3m+A)+r$ can not be a square.
\end{enumerate}

Thus, $3pn+R$ is neither a square nor twice a square for any value of $n$, $p$, and $R$. This completes the proof of Theorem \ref{singlemodcong}. \qed

\section{Proof of Theorem \ref{thm:conj}}\label{sec:conj}

    We note that the first congruence \eqref{conjp-1} follows immediately from Theorem \ref{thm:k2}. We prove the second congruence \eqref{conjp-2} next.

Using \eqref{dis1byf1^2}, we have
\begin{align*}
\sum_{n=0}^\infty \overline{OPT}_{2^i\cdot r}(n)q^n&=\dfrac{f_2^{3\cdot 2^i\cdot r}}{f_4^{2^i\cdot r}}\left(\dfrac{1}{f_1^2}\right)^{2^i\cdot r}=\dfrac{f_2^{3\cdot 2^i\cdot r}}{f_4^{2^i\cdot r}}\left(\frac{f_8^5}{f_2^5 f_{16}^2}+2q\frac{f_4^2 f_{16}^2}{f_2^5 f_8}\right)^{2^i\cdot r}\\
&=\sum_{k=0}^{2^i\cdot r}2^k\binom{2^i\cdot r}{k}q^k\dfrac{f_8^{5\cdot 2^i\cdot r-6 k}}{f_2^{2^{i+1}\cdot r}f_4^{2^i\cdot r-2k}f_{16}^{2^{i+1}\cdot r-4k}},
\end{align*}
from which
\begin{align*}
\sum_{n=0}^\infty \overline{OPT}_{2^i\cdot r}(2n+1)q^n
&=\sum_{k=0}^{2^{i-1}\cdot r-1/2}2^{2k+1}\binom{2^i\cdot r}{2k+1}q^k\dfrac{f_4^{5\cdot 2^i\cdot r-12k-6}}{f_1^{2^{i+1}\cdot r}f_2^{2^i\cdot r-4k-2}f_{8}^{2^{i+1}\cdot r-8k-4}}\\
&=\sum_{k=0}^{2^{i-1}\cdot r-1/2}2^{2k+1}\binom{2^i\cdot r}{2k+1}q^k\dfrac{f_4^{5\cdot 2^i\cdot r-12k-6}}{f_2^{2^i\cdot r-4k-2}f_{8}^{2^{i+1}\cdot r-8k-4}}\cdot \left(\dfrac{1}{f_1^2}\right)^{2^{i}\cdot r}\\
&=\sum_{k=0}^{2^{i-1}\cdot r-1/2}\sum_{t=0}^{2^{i}\cdot r}2^{2k+t+1}\binom{2^i\cdot r}{2k+1}\binom{2^i\cdot r}{t}\\ &\quad \times q^{k+t}\dfrac{f_4^{5\cdot 2^i\cdot r-12k+2t-6} f_8^{3\cdot 2^i\cdot r+8k-6t+4}}{f_2^{3\cdot 2^{i+1}\cdot r-4k-2}f_{16}^{2^{i+1}\cdot r-4t}},
\end{align*}
which is equivalent to
\begin{align*}
\sum_{n=0}^\infty \overline{OPT}_{2^i\cdot r}(2n+1)q^n
&=\sum_{\alpha=0,\beta=0}^1\sum_{k=0}^{2^{i-2}\cdot r-1/4-\alpha/2}\sum_{t=0}^{2^{i-1}\cdot r-\beta/2}2^{4k+2t+2\alpha+\beta+1}\binom{2^i\cdot r}{4k+2\alpha+1}\\
&\quad\times\binom{2^i\cdot r}{2t+\beta}q^{2k+2t+\alpha+\beta}\dfrac{f_4^{5\cdot 2^i\cdot r-24k+4t-12\alpha+2\beta-6} f_8^{3\cdot 2^i\cdot r+16k-12t+8\alpha-6\beta+4}}{f_2^{3\cdot 2^{i+1}\cdot r-8k-4\alpha-2}f_{16}^{2^{i+1}\cdot r-8t-4\beta}}.
\end{align*}
From the above identity, we extract the terms that involve odd exponents of $q$, 
\begin{align}
\sum_{n=0}^\infty \overline{OPT}_{2^i\cdot r}(4n+3)q^n
&=\sum_{\substack{\alpha,\beta\\\alpha+\beta=1}}^1\sum_{k=0}^{2^{i-2}\cdot r-1/4-\alpha/2}\sum_{t=0}^{2^{i-1}\cdot r-\beta/2}2^{4k+2t+2\alpha+\beta+1}\binom{2^i\cdot r}{4k+2\alpha+1}\notag\\
&\quad\times\binom{2^i\cdot r}{2t+\beta}q^{k+t}\dfrac{f_2^{5\cdot 2^i\cdot r-24k+4t-12\alpha+2\beta-6} f_4^{3\cdot 2^i\cdot r+16k-12t+8\alpha-6\beta+4}}{f_1^{3\cdot 2^{i+1}\cdot r-8k-4\alpha-2}f_{8}^{2^{i+1}\cdot r-8t-4\beta}}.\label{Correct 1}
\end{align}
By Lemma \ref{lem1}, we have
\begin{align}
2^{4k+2t+2\alpha+\beta+1}\binom{2^i\cdot r}{4k+2\alpha+1}
\binom{2^i\cdot r}{2t+\beta}&\equiv 0 \pmod{2^{i+3}}\label{Correct 2}
\end{align}
for the tuples $(\alpha,\beta)=(0,1)$ and $(1,0)$ except when $k=0$, $t=0$, and $\beta=0$. So, we consider 
\begin{align*}
    2^{4\cdot 0+2\cdot 0+2\cdot 1+0+1}\binom{2^i\cdot r}{4\cdot 0+2\cdot 1+1}
\binom{2^i\cdot r}{2\cdot 0+0}&= 2^3\cdot \dfrac{2^i\cdot r\left(2^{i}\cdot r-1\right)\left(2^{i-1}\cdot r-1\right)}{3}\\
&\equiv 0 \pmod{2^{i+3}}.
\end{align*}
Therefore, \eqref{Correct 1} and \eqref{Correct 2} give \eqref{conjp-2}. 

The third congruence \eqref{conjp-3} follows from a result analogous to a result of Keister, Vary and Sellers \cite[Theorem 9]{KeisterSellersVary}. \qed

\begin{lemma}\label{lem:new}
    Let $k=2^mr$, $m> 0$ and $r$ be odd, then for all $n\geq 1$ we have
\[
        \opt_{2^mr}(n)\equiv \begin{cases} 2^{m+1}\pmod {2^{m+2}} &\text{if $n$ is a square, twice a square or four times a square,}\\0 \pmod {2^{m+2}} & \text{otherwise}.
    \end{cases}
    \]
\end{lemma}
\begin{proof}
We prove the result by induction on $m$. The base case $m=1$ is given first, where we show
\[
        \opt_{2r}(n)\equiv \begin{cases} 4\pmod 8\quad \text{if $n$ is a square, twice a square or four times a square,}\\0 \pmod8 \quad \text{otherwise}.
    \end{cases}
    \]
    We have
    \begin{align*}
        \sum_{n\geq 0}\opt_{2r}(n)q^n&=\left(\left(\sum_{n\geq 0}\opt_1(n)q^n\right)^2\right)^r\\
        &=\left(\left(1+\sum_{\mycom{n>0}{\text{$n$ is a square}}}\opt_1(n)q^n+\sum_{\mycom{n>0}{\text{$n$ is not a square}}}\opt_1(n)q^n\right)^2\right)^r.
    \end{align*}
We expand the above square and look at each term separately:
\begin{multline*}
    1+\left(\sum_{\mycom{n>0}{\text{$n$ is a square}}}\opt_1(n)q^n\right)^2+\underbrace{\left(\sum_{\mycom{n>0}{\text{$n$ is not a square}}}\opt_1(n)q^n\right)^2}_{A}+2\left(\sum_{\mycom{n>0}{\text{$n$ is a square}}}\opt_1(n)q^n\right)\\
    +\underbrace{2\left(\sum_{\mycom{n>0}{\text{$n$ is not a square}}}\opt_1(n)q^n\right)}_{B}+\underbrace{2\left(\sum_{\mycom{n>0}{\text{$n$ is a square}}}\opt_1(n)q^n\right)\left(\sum_{\mycom{n>0}{\text{$n$ is not a square}}}\opt_1(n)q^n\right)}_{C}.
\end{multline*}

First, we look at the term labelled $A$:
\begin{align*}
    A&=\left(\sum_{\mycom{n>0}{n=2\times \text{ a square}}}\opt_1(n)q^n+\sum_{\mycom{n>0}{n\neq 2\times \text{ a square or a square}}}\opt_1(n)q^n\right)^2\\
    &\equiv \left(\sum_{\mycom{n>0}{n=2\times \text{a square}}}\opt_1(n)q^n\right)^2 \pmod 8,
\end{align*}
where we have used Lemma \ref{lem2} and \ref{lem3} to obtain the last step. Similarly, another application of Lemma \ref{lem2} would give us
\[
B\equiv 2\left(\sum_{\mycom{n>0}{n=2\times \text{a square}}}\opt_1(n)q^n\right).
\]
Finally, the term labelled $C$ gives us
\begin{align*}
    C&=2\left(\sum_{\mycom{n>0}{n=2\times \text{a square}}}\opt_1(n)q^n+\sum_{\mycom{n>0}{n\neq 2\times \text{a square or a square}}}\opt_1(n)q^n\right)\left(\sum_{\mycom{n>0}{n= \text{a square}}}\opt_1(n)q^n\right)\\
    &\equiv 2\left(\sum_{\mycom{n>0}{n=2\times \text{a square}}}\opt_1(n)q^n\right)\left(\sum_{\mycom{n>0}{n= \text{a square}}}\opt_1(n)q^n\right) \pmod 8\\
    &\equiv 0 \pmod 8.
\end{align*}
Here we have applied Lemma \ref{lem2} to obtain each of the two steps.

Collecting all of the above together, we obtain
\begin{align*}
    \sum_{n\geq 0}\opt_{2r}(n)q^n&=\Biggl(1+\left(\sum_{\mycom{n>0}{n= \text{a square}}}\opt_1(n)q^n\right)^2+2\left(\sum_{\mycom{n>0}{n=\text{a square}}}\opt_1(n)q^n\right)\\
    &\quad +\left(\sum_{\mycom{n>0}{n=2\times \text{a square}}}\opt_1(n)q^n\right)^2+2\left(\sum_{\mycom{n>0}{n=2\times \text{a square}}}\opt_1(n)q^n\right)\Biggl)^r\pmod 8\\
    &=\Biggl(1+\left(\sum_{n\geq 1}\opt_1(n^2)q^{n^2}\right)^2+\left(\sum_{n\geq 1}\opt_1(2n^2)q^{2n^2}\right)^2\\
    &\quad +2\left(\sum_{n\geq 1}\opt_1(n^2)q^{n^2}\right)+2\left(\sum_{n\geq 1}\opt_1(2n^2)q^{2n^2}\right)\Biggl)^r\pmod 8.
\end{align*}
Applying Lemma \ref{lem2} in the above we obtain
\begin{equation*}
     \sum_{n\geq 0}\opt_{2r}(n)q^n=\left(1+4\sum_{n\geq 1}q^{n^2}+4\sum_{n\geq 1}q^{2n^2}+4\left(\sum_{n\geq 1}q^{n^2}\right)^2+4\left(\sum_{n\geq 1}q^{2n^2}\right)^2\right)^r \pmod 8.
\end{equation*}
Using the fact that $(q^{n_1}+q^{n_2}+\cdots )^2=(q^{2n_1}+q^{2n_2}+\cdots )+2(q^{n_1+n_2}+\cdots)$ in the above, we obtain
\begin{align*}
    \sum_{n\geq 0}\opt_{2r}(n)q^n&=\Biggl(1+4\sum_{n\geq 1}q^{n^2}+4\sum_{n\geq 1}q^{2n^2}+4\left(\sum_{n\geq 1}q^{2n^2}+2\sum_{\mycom{n_1,n_2>0}{n_1\neq n_2}}q^{n_1^2+n_2^2}\right)\\
    &\quad +4\left(\sum_{n\geq 1}q^{4n^2}+2\sum_{\mycom{n_1,n_2>0}{n_1\neq n_2}}q^{2n_1^2+2n_2^2}\right)\Biggl)^r\pmod 8\\
    &=\left(1+4\left(\sum_{m\geq 1}q^{m^2}+\sum_{n\geq 1}q^{2n^2}+\sum_{k\geq 1}q^{4k^2}\right)\right)^r\pmod8\\
    &=\sum_{j\geq 0}\binom{r}{j}4^j\left(\sum_{m\geq 1}q^{m^2}+\sum_{n\geq 1}q^{2n^2}+\sum_{k\geq 1}q^{4k^2}\right)^j\\
    &\equiv 1+4\left(\sum_{m\geq 1}q^{m^2}+\sum_{n\geq 1}q^{2n^2}+\sum_{k\geq 1}q^{4k^2}\right)\pmod 8.
\end{align*}
The last step follows since $r$ is odd. This proves the base case.

For the induction step, we assume
\[
        \opt_{2^mr}(n)\equiv \begin{cases} 2^{m+1}\pmod {2^{m+2}}& \text{if $n$ is a square, twice a square or four times a square,}\\0 \pmod {2^{m+2}} & \text{otherwise}.
    \end{cases}
    \]
    and show that
    \[
        \opt_{2^{m+1}r}(n)\equiv \begin{cases} 2^{m+2}\pmod {2^{m+3}}& \text{if $n$ is a square, twice a square or four times a square,}\\0 \pmod {2^{m+3}} & \text{otherwise}.
    \end{cases}
    \]
    We have
    \begin{align*}
        \sum_{n\geq 0}&\opt_{2^{m+1}r}(n)q^n\\
        &=\left(\sum_{n\geq 0}\opt_{2^{m}r}(n)q^n\right)^2\\
        &=\left(1+\sum_{\mycom{n> 0}{n\neq 1\times \text{or}~2\times \text{or}~4\times \text{square}}}\opt_{2^{m}r}(n)q^n+\sum_{\mycom{n> 0}{n= 1\times \text{or}~2\times \text{or}~4\times \text{square}}}\opt_{2^{m}r}(n)q^n\right)^2\\
        &=1+\left(\sum_{\mycom{n> 0}{n\neq 1\times \text{or}~2\times \text{or}~4\times \text{square}}}\opt_{2^{m}r}(n)q^n\right)^2+\left(\sum_{\mycom{n> 0}{n= 1\times \text{or}~2\times \text{or}~4\times \text{square}}}\opt_{2^{m}r}(n)q^n\right)^2\\
        &\quad +2 \left(\sum_{\mycom{n> 0}{n\neq 1\times \text{or}~2\times \text{or}~4\times \text{square}}}\opt_{2^{m}r}(n)q^n\right)+2\left(\sum_{\mycom{n> 0}{n= 1\times \text{or}~2\times \text{or}~4\times \text{square}}}\opt_{2^{m}r}(n)q^n\right)\\
        &\quad +2\left(\sum_{\mycom{n>0}{n\neq 1\times \text{or}~2\times \text{or}~4\times \text{square}}}\opt_{2^{m}r}(n)q^n\right) \times \left(\sum_{\mycom{n> 0}{n= 1\times \text{or}~2\times \text{or}~4\times \text{square}}}\opt_{2^{m}r}(n)q^n\right).
    \end{align*}
    Using a similar technique like the base case and using the induction hypothesis we arrive at
    \begin{align*}
       \sum_{n\geq 0}\opt_{2^{m+1}r}(n)q^n&\equiv 1+2\sum_{n\geq 1}\opt_{2^mr}(n^2)q^{n^2}+2\sum_{n\geq 1}\opt_{2^mr}(2n^2)q^{2n^2}\\
       &\quad +2\sum_{n\geq 1}\opt_{2^mr}(4n^2)q^{4n^2}+\Biggl(\sum_{n\geq 1}\opt_{2^mr}(n^2)q^{n^2}\\
       &\quad +\sum_{n\geq 1}\opt_{2^mr}(2n^2)q^{2n^2}
        + \sum_{n\geq 1}\opt_{2^mr}(4n^2)q^{4n^2}\Biggl)^2 \pmod{2^{m+3}}\\
       &\equiv 1+2\Biggl(\sum_{n\geq 1}\opt_{2^mr}(n^2)q^{n^2}+\sum_{n\geq 1}\opt_{2^mr}(2n^2)q^{2n^2}+\\
       & \quad +\sum_{n\geq 1}\opt_{2^mr}(4n^2)q^{4n^2}\Biggl) \pmod{2^{m+3}}.
    \end{align*}
    From the induction hypothesis, the coefficients of the last term are congruent to $2^{m+1}$ or $2^{m+1}+2^{m+2}$ or $2^{m+1}+2^{m+2}+2^{m+3} \pmod{2^{m+3}}$, and with the factor of $2$ in front of it, we finally arrive at
    \begin{equation}
         \sum_{n\geq 0}\opt_{2^{m+1}r}(n)q^n\equiv 1+2^{m+2}\left(\sum_{n\geq 1}q^{n^2}+\sum_{m\geq 1}q^{2m^2}+\sum_{k\geq 1}q^{4k^2}\right) \pmod{2^{m+3}}.
    \end{equation}
    This completes the induction.
\end{proof}

The third congruence \eqref{conjp-3} follows easily from Lemma \ref{lem:new}. Clearly, $8n+5$ is not twice or four times a square.  Also, $8n+5$ is not a square as it is odd and odd squares are $\equiv 1 \pmod 8$.

\section{Proof of Theorem \ref{thm:KSV:mod3}}\label{sec:KSV:mod3}

    We have
    \begin{align*}
        \sum_{n=0}^\infty \overline{OPT}_{3^i}(n)q^n&=\dfrac{f_2^{3\cdot 3^i}}{f_1^{2\cdot 3^i}f_4^{3^i}}=\dfrac{f_{-1}^{3^i}}{f_1^{3^i}},
    \end{align*}
    where for positive odd integers $k$, we take $f_{-k}:=\left(-q^k;-q^k\right)_\infty$ and it is known that $f_{-k}=f_{2k}^3/\left(f_kf_{4k}\right)$. In the above identity, we invoke the 3-dissections \eqref{3df1^3} and \eqref{3d1/f1^3} and then use binomial expansions to arrive at
    \begin{align*}
        &\sum_{n=0}^\infty \overline{OPT}_{3^i}(n)q^n\\
        &=\left(a_{-3}f_{-3}+3q f_{-9}^3\right)^{3^{i-1}}\times\left(a_3^2\frac{f_9^3}{f_3^{10}}+3 a_3 q \frac{f_9^6}{f_3^{11}}+9 q^2\frac{f_9^9}{f_3^{12}}\right)^{3^{i-1}}\\
        &=\left(\sum_{t=0}^{3^{i-1}}3^t\binom{3^{i-1}}{t}q^ta_{-3}^{3^{i-1}}f_{-3}^{3^{i-1}-t}f_{-9}^{3t}\right)\\&\quad \times \left(\sum_{\substack{\ell=0,m=0,r=0\\ \ell+m+r=3^{i-1}}}^{3^{i-1}}3^{m+2r}\binom{3^{i-1}}{\ell,m,r}q^{m+2r}a_{3}^{3^{2\ell+m}}\dfrac{f_9^{3\ell+9m+6r}}{f_3^{10\ell+11m+12r}}\right)\\
        &=\left(\sum_{\alpha=0}^{2}\sum_{t=0}^{3^{i-2}-\alpha/3}3^{3t+\alpha}\binom{3^{i-1}}{3t+\alpha}q^{3t+\alpha}a_{-3}^{3^{i-1}}f_{-3}^{3^{i-1}-3t-\alpha}f_{-9}^{9t+3\alpha}\right) \\
        & \quad \times \Bigg(\sum_{\beta=0,\gamma=0}^{2}\sum_{\substack{\ell=0,m=0,r=0\\ \ell+3m+3r+\beta+\gamma=3^{i-1}}}^{3^{i-1},~ 3^{i-2}-\beta/3,~ 3^{i-2}-\gamma/3} 3^{3m+6r+\beta+2\gamma}\binom{3^{i-1}}{\ell,3m+\beta,3r+\gamma}\\
        &\quad 
        \times q^{3m+6r+\beta+2\gamma} a_{3}^{3^{2\ell+3m+\beta}} \dfrac{f_9^{3\ell+27m+18r+9\beta+6\gamma}}{f_3^{10\ell+33m+36r+11\beta+12\gamma}}\Bigg).
    \end{align*}
 From the above product of sums, extracting the terms that involve $q^{3n+2}$, we obtain
\begin{align}
         &\sum_{n=0}^\infty \overline{OPT}_{3^i}(3n+2)q^{n}\notag\\
         &= \sum_{\substack{\alpha,\beta,\gamma\\ \alpha+\beta+2\gamma~\equiv~ 2~(\textup{mod}~3)}}^{2}\sum_{\substack{t, \ell,m,r\\ \ell+3m+3r+1=3^{i-1}}}^{3^{i-2}-\alpha/3,~3^{i-1},~ 3^{i-2}-\beta/3,~ 3^{i-2}-\gamma/3} 3^{3t+3m+6r+\alpha+\beta+2\gamma} \binom{3^{i-1}}{3t+\alpha}\notag\\
         &\quad\times \binom{3^{i-1}}{\ell,3m+\beta,3r+\gamma}q^{t+m+2r+\alpha/3+\beta/3+2\gamma/3-2/3}a_{-1}^{3^{i-1}} a_{1}^{3^{2\ell+3m+\beta}}f_{-1}^{3^{i-1}-3t-\alpha}f_{-3}^{9t+3\alpha} \notag\\
         &\quad\times\dfrac{f_3^{3\ell+27m+18r+9\beta+6\gamma}}{f_1^{10\ell+33m+36r+11\beta+12\gamma}}.\label{OPT3i:3n+2}
\end{align}
Now, it remains to be examined if
\begin{align}
    3^{3t+3m+6r+\alpha+\beta+2\gamma}\binom{3^{i-1}}{3t+\alpha}\binom{3^{i-1}}{\ell,3m+\beta,3r+\gamma}&\equiv0\pmod{3^{i+1}},\label{Mixed:KSV}
\end{align}
for the tuples \[(\alpha,\beta,\gamma)\in \{(0,0,1),(0,1,2),(0,2,0),(1,0,2),(1,1,0),(1,2,1),(2,0,0),(2,1,1),(2,2,2)\}.\] This follows immediately from Lemma \ref{modified-ksv} for $1\leq t\leq 3^{i-1}-1$. Next, for the case $t=0$ and $\alpha=0$, we need to check whether
\begin{align}
     3^{3m+6r+\beta+2\gamma}\binom{3^{i-1}}{\ell,3m+\beta,3r+\gamma}\equiv0\pmod{3^{i+1}}. \label{Mixed:KSV2}
\end{align}
This can easily be seen as we have $\alpha=0$ and hence $\beta+2\gamma$ must be at least 2. Also, $\binom{3^{i-1}}{\ell,3m+\beta,3r+\gamma}$ is divisible by $3^{i-1}$. Therefore, \eqref{OPT3i:3n+2}, \eqref{Mixed:KSV} and \eqref{Mixed:KSV2} together give \eqref{KSV:mod3:Cong1}. \qed

\section{Proof of Theorems \ref{thm:mf-1}}\label{sec:mf}
    From \cite[Eq. (44)]{SSS} we have
    \[
    \sum_{n\geq 0}\opt_3(8n+1)q^n\equiv 6f_1f_2 \pmod 8.
    \]
    This implies
    \begin{equation}\label{eq:ass}
    \sum_{n\geq 0}\opt_3(8n+1)q^{8n+1}\equiv 6\eta(8z)\eta(16z) \pmod 8.
    \end{equation}
    Let $\eta(8z)\eta(16z):=\sum_{n=1}^\infty a(n)q^n$, then $a(n)=0$ if $n\not\equiv 1\pmod 8$ for all $n\geq 0$. This implies
    \begin{equation}
        \opt_3(8n+1)\equiv 6 a(8n+1) \pmod 8.
    \end{equation}
    From Theorem \ref{thm:ono1} we have $\eta(8z)\eta(16z)\in S_1(\Gamma_0(128), \chi_1)$, where $\chi_1$ is a Nebentypus character and is given by $\chi_1(\tikz\draw[black,fill=black] (0,0) circle (.5ex);)=\left(\frac{-1\cdot2^7}{\tikz\draw[black,fill=black] (0,0) circle (.5ex);}\right)$. 
    
    Since $\eta(8z)\eta(16z)$ is a Hecke eigenform (see \cite[pp. 4854]{Martin}), so equation \eqref{hecke} implies
    \[
    \eta(8z)\eta(16z)|T_p=\sum_{n=1}^\infty \left(a(pn)+\chi_1(p)a\left(\frac{n}{p}\right)\right)q^n=\lambda(p)\sum_{n=1}^\infty a(n)q^n.
    \]
    This implies
    \begin{equation}\label{eq:as}
        a(pn)+\chi_1(p)a\left(\frac{n}{p}\right)=\lambda(p)a(n).
    \end{equation}
    Putting $n=1$ we note that $a(1)=1$, so $\lambda(p)=a(p)$ and since $a(p)=0$ for all $p\not\equiv 1 \pmod 8$ we have $\lambda(p)=0$.
    From \eqref{eq:as} we obtain,
    \begin{equation}\label{eq:b}
        a(pn)+\chi_1(p)a\left(\frac{n}{p}\right)=0.
    \end{equation}
    In equation \eqref{eq:b} setting $n=pn+r$ we obtain, for all $n\geq 0$ and $p\nmid r$,
    \begin{equation}\label{eq:44}
        a(p^2n+pr)=0,
    \end{equation}
    and replacing $n$ by $pn$ in \eqref{eq:b} we obtain
    \begin{equation}\label{eq:45}
        a(p^2n)\equiv -\chi_1(p)a(n) \pmod 4.
    \end{equation}
    Let $A(n):=a(8n+1)$, and let $p$ be a prime such that $p\not\equiv 1\pmod 8$. Replacing $n$ by $8n-pr+1$ in \eqref{eq:44} we obtain
    \begin{equation}\label{eq:46}
        A\left(p^2n+\frac{p^2-1}{8}+pr\frac{1-p^2}{8}\right)=0.
    \end{equation}
    Setting $n=8n+1$ in \eqref{eq:45}, we have
    \begin{equation}\label{eq:47}
   A\left(p^2n+\frac{p^2-1}{8}\right)\equiv -\chi_1(p)A(n) \pmod 4.     
    \end{equation}
    Since $p\geq 3$ is a prime, so $8|(1-p^2)$ and $\gcd\left(\frac{1-p^2}{8},p\right)=1$. So, $r$ runs over a residue system excluding the multiples of $p$, and so does $\frac{1-p^2}{8}r$. We can rewrite \eqref{eq:46} as
    \begin{equation}\label{eq:48}
                A\left(p^2n+\frac{p^2-1}{8}+pj\right)\equiv0 \pmod 4,
    \end{equation}
    where $p\nmid j$.

    For primes $p_i\geq 3$ such that $p_i\not\equiv 1 \pmod 8$ we can use \eqref{eq:47} repeatedly to obtain
    \begin{equation}\label{eq:49}
           A\left(p_1^2p_2^2\ldots p_{k}^2n+\frac{p_1^2p_2^2\ldots p_{k}^2-1}{8}\right)\equiv (-1)^k\chi_1(p_1)\chi_1(p_2)\cdots\chi_1(p_{k})A(n) \pmod 4. 
    \end{equation}
This needs the observation
\[
p_1^2p_2^2\ldots p_k^2n+\frac{p_1^2p_2^2\ldots p_k^2-1}{8}=p_1^2\left(p_2^2\ldots p_k^2n+\frac{p_2^2\ldots p_k^2-1}{8}\right)+\frac{p_1^2-1}{8}.
\]
Let $j\not\equiv 0\pmod{p_{k+1}}$, then \eqref{eq:48} and \eqref{eq:49} gives us
\begin{equation}\label{eq:asss}
     A\left(p_1^2p_2^2\ldots p_{k+1}^2n+\frac{p_1^2p_2^2\ldots p_{k+1}^2-1}{8}+p_1^2p_2^2\ldots p_k^2p_{k+1}j\right)\equiv 0 \pmod 4. 
\end{equation}
Finally, equations \eqref{eq:asss} and \eqref{eq:ass} gives us the desired result.

\section{Proof of Theorem \ref{thm:m2}}\label{proof:density}
\textbf{Case (i):} When $p=2$.

Recall
\begin{equation}\label{eq:1}
    \sum_{n\geq 0}\opt_3(n)q^n=\frac{f_2^9}{f_1^6f_4^3}.
\end{equation}
Rewriting this in terms of eta-quotients we have
\begin{equation}\label{eq:2}
     \sum_{n\geq 0}\opt_3(n)q^{24n}=\frac{\eta^9(48z)}{\eta^6(24z)\eta^3(96z)}.
\end{equation}
Let
\[
A_2(z)=\frac{\eta^{2}(24z)}{\eta(48z)}.
\]
By the binomial theorem, we have
\[
A_2^{2^k}(z)\equiv 1 \pmod{2^{k+1}}.
\]
Now, define
\[
C_{2,k}(z):=\frac{\eta^9(48z)}{\eta^6(24z)\eta^3(96z)}A_2^{2^k}(z)=\frac{\eta^{2^{k+1}-6}(24z)}{\eta^{2^k-9}(48z)\eta^3(96z)}.
\]
By equation \eqref{eq:1} and \eqref{eq:2} we have
\begin{equation}\label{eq:a}
    C_{2,k}(z)\equiv \sum_{n\geq 0}\opt_3(n)q^{24n} \pmod{2^{k+1}}
\end{equation}
By Theorem \ref{thm:ono1} $C_{2,k}(z)$ is a form of weight $2^{k-1}$ on $\Gamma_0(768)$. The cusps of $\Gamma_0(768)$ are represented by fractions $\dfrac{c}{d}$, where $d|768$ and $\gcd(c,d)=1$. $C_{2,k}(z)$ is holomorphic at a cusp $\dfrac{c}{d}$ if and only if
\begin{equation}\label{eq:3}
    \frac{\gcd(d,24)^2}{24}(2^{k+1}-6)+    \frac{\gcd(d,48)^2}{48}(9-2^k)-    3\frac{\gcd(d,96)^2}{96}\geq 0.
\end{equation}

Using Mathematica, we verify the inequality \eqref{eq:3} for all divisors of $768$. Inequality \eqref{eq:3} holds true for $k\geq3$. So, by Theorem \ref{thm:ono1}, $C_k(z)\in M_{2^{k-1}}\left(\Gamma_0(768),\left(\dfrac{2^{2^{k+1}+6}\cdot3^{2^k}}{\tikz\draw[black,fill=black] (0,0) circle (.5ex);}\right)\right)$ for $k\geq3$. By Theorem \ref{thm:ono2}, the Fourier coefficients of $C_k(z)$ are almost always divisible by $m=2^{k}$ for $k\geq3$. By \eqref{eq:a}, the same is true for $\opt_3(n)$ and thus $\opt_3(n)$ is dense for $2^k$ for $k\geq3$. But, if any integer is divisible by $2^k$ for $k\geq3$, then it is also divisible by 4 and 2 and hence the result follows for $p=2$.





\textbf{Case (ii):} $p\geq5$

  Define 
  \[A_3(z):=\frac{\eta(24z)}{\eta(24pz)}.
  \]
  Working modulo $p^k$, we have  
  \[
    A_3^{p^k}(z)=\frac{\eta^{p^{k+1}}(24z)}{\eta^{p^k}(24pz)}\equiv 1 \pmod{p^{k+1}},
    \]
    and
    \[
    C_{p,k}(z)=\frac{\eta^9(48z)}{\eta^6(24z)\eta^3(96z)}A^{p^k}(z).
    \]
    Modulo $p^{k}$, we have
    \[
    C_{p,k}(z)\equiv \frac{\eta^9(48z)}{\eta^6(24z)\eta^3(96z)}\equiv \sum_{n\geq 0}\opt_3(n)q^{24n}\pmod{p^k}.
    \]
    The weight of the eta-quotient $C_{p,k}(z)$ is $\frac{p^k}{2}(p-1)$. Let the level of the eta-quotient be $96up$, where $u$ is the smallest integer satisfying
    \[
    (p^{k+1}-6)\frac{96u}{24}+ 9\frac{96u}{48}-3\frac{96u}{96}-p^k\frac{96u}{24p}\equiv 0 \pmod{24}.
    \]
    This is equivalent to
    \begin{equation}\label{eq:32}
      ( 4p^{k-1}(p^{2}-1)-9)u\equiv 0 \pmod{24}.
    \end{equation}
    For all primes $p\geq 5$ we have $3|p^{2}-1$, so from \eqref{eq:32} we can conclude that $u=8$.

    By Theorem \ref{thm:ono1}, the cusps of $\Gamma_0(768p)$ are given by $\dfrac{c}{d}$ where $d\mid768p$ and $\gcd(c,d)=1$. So, $C_{p,z}(z)$ is holomorphic at a cusp $\dfrac{c}{d}$ if and only if
    \begin{equation}\label{eq:33}
        (p^{k+1}-6)\frac{\gcd(d,24)^2}{24}+9\frac{\gcd(d,48)^2}{48}-3\frac{\gcd(d,96)^2}{96}-p^k\frac{\gcd(d,24p)^2}{24p}\geq 0.
    \end{equation}
    We need to check whether \eqref{eq:33} is true for all $d|768p$. We verify this with the help of the following table.
    \begin{center}
	{\setlength{\extrarowheight}{18pt}\begin{tabular}{|c|c|c|c|c|}
	\hline 
        $d$ & $\dfrac{\gcd(d,24)^2}{24}$ &$\dfrac{\gcd(d,48)^2}{48}$&$\dfrac{\gcd(d,96)^2}{96}$&$\dfrac{\gcd(d,24p)^2}{24p}$\\
        \hline
        $2^r\cdot 3^s$,where $0\leq r
        \leq 3$ and $0\leq s
        \leq 1$ & $\dfrac{d^2}{24}$ & $\dfrac{d^2}{48}$ & $\dfrac{d^2}{96}$ & $\dfrac{d^2}{24p}$\\
         \hline
         $2^r$,where $5\leq r
        \leq 8$ & $\dfrac{8}{3}$ & $\dfrac{16}{3}$ & $\dfrac{32}{3}$ & $\dfrac{8}{3p}$\\
       \hline
        $2^r\cdot3$,where $5\leq r
        \leq 8$ & $24$ & $48$ & $96$ & $\dfrac{24}{p}$\\
         \hline
          $16$ & $\dfrac{8}{3}$ & $\dfrac{16}{3}$ & $\dfrac{8}{3}$ & $\dfrac{8}{3p}$\\
         \hline
         $48$ & $24$ & $48$ & $24$ & $\dfrac{24}{p}$\\
         \hline
        $2^r\cdot3^s\cdot p$,where $0\leq r
        \leq 3$ and $0\leq s
        \leq 1$ & $\dfrac{d^2}{24p}$ & $\dfrac{d^2}{48p}$ & $\dfrac{d^2}{96p}$ & $\dfrac{d^2}{24p}$\\
         \hline
         $2^r\cdot p$,where $5\leq r
        \leq 8$ & $\dfrac{8}{3}$ & $\dfrac{16}{3}$ & $\dfrac{32}{3}$ & $\dfrac{8p}{3}$\\
         \hline
        $2^r\cdot 3 \cdot p$,where $5\leq r
        \leq 8$ & $24$ & $48$ & $96$ & $24p$\\
         \hline
          $16\cdot p$ & $\dfrac{8}{3}$ & $\dfrac{16}{3}$ & $\dfrac{8}{3}$ & $\dfrac{8p}{3}$\\
         \hline
         $48\cdot p$ & $24$ & $48$ & $24$ & $24p$\\
         \hline
	\end{tabular}}
\end{center}

\noindent Now, using Theorems \ref{thm:ono1} and \ref{thm:ono2} we are done. 

\section{Concluding Remarks}\label{sec:conc}

\begin{enumerate}
\item Theorem \ref{thm:equality} looks ripe for a combinatorial proof. We leave this as an open problem.
\item Can we extend Theorem \ref{singlemodcong} to the $\opt_{2k+1}(n)$ function?
    \item There are other candidates for a result like Theorem \ref{thm:mf-1}. For instance, using some of the identities in the proof of Theorem 10 in \cite{SSS} there might be some results for $\opt_4(n)$.
    \item Based on strong numerical evidence, we pose the following conjectures.
\begin{conjecture}
    For all $i,j\geq 1$ and $k$ not a multiple of $2$ or $3$, we have
\[
    \opt_{3^i\cdot 2^j\cdot k}(3n+2)\equiv 0\pmod {3^{i+1}\cdot 2^{j+2}}.
\]
\end{conjecture}

\begin{conjecture}
    For all $i\geq 1$ and $j$ not a power of $2$, nor a multiple of $2$ or divisible by $3$, we have
\begin{align*}
  \opt_{3^i\cdot j}(3n+2)&\equiv 0\pmod {3^{i+1}\cdot 2}.
\end{align*}
\end{conjecture}

\begin{conjecture}
    For all $i,j\geq 1$ and $k$ not a multiple of $2$ or $3$, we have
\[
    \opt_{3^i\cdot 2^j\cdot k}(3n+1)\equiv 0\pmod {3^{i}\cdot 2^{j+1}}.
\]
\end{conjecture}

\begin{conjecture}
    For all $i\geq 1$ and $j$ not a power of $2$, nor a multiple of $2$ or divisible by $3$, we have
\begin{align*}
  \opt_{3^i\cdot j}(3n+1)&\equiv 0\pmod {3^{i}\cdot 2}.
\end{align*}
\end{conjecture}
\noindent It was communicated to us that these conjectures have been recently proven by Qi, Sang, Yao, and Zhou \cite{QiSangYaoZhou}. In addition, recently Keerthana, Ananya, and Ranganatha \cite{arxivnew} also announced a proof of these conjectures.
\end{enumerate}

\section*{Acknowledgements}

The authors are extremely grateful to Professor James Sellers. Discussions with Professor Sellers led the authors to the  development of Theorems \ref{thm:equality} and \ref{singlemodcong}. The second author was partially supported by a Start-Up Grant from Ahmedabad University (Reference No. URBSASI24A5). The authors also thank the anonymous referee for helpful suggestions.

\section*{Data availability}
Data sharing is not applicable to this article.

\section*{Declaration}
\textbf{Conflict of interest:} The authors declare that they have no conflict of interest.

\end{document}